\documentclass{article}
\usepackage{amsmath,amssymb,amsfonts,amsthm}
\title{\bf M\"obius transforms and cyclic equations}

\author{Kurt Girstmair\\
Institut f\"ur Mathematik \\
Universit\"at Innsbruck   \\
Technikerstr. 13/7        \\
A-6020 Innsbruck, Austria \\
Kurt.Girstmair@uibk.ac.at}

\date{}

\makeatletter
\let\@@maketitle=\maketitle
\def\maketitle{\def\thispagestyle##1{\relax}\@@maketitle}
\makeatother
\newtheorem{theorem}{Theorem}

\newenvironment{remark}{{\em Remark.}}{}

\def\BE{\begin{equation}}
\def\EE{\end{equation}}
\def\BD{\begin{displaymath}}
\def\ED{\end{displaymath}}
\def\BA{\begin{array}}
\def\EA{\end{array}}
\def\BEA{\begin{eqnarray*}}
\def\EEA{\end{eqnarray*}}

\def\Z{\mathbb Z}
\def\Q{\mathbb Q}
\def\R{\mathbb R}
\def\C{\mathbb C}

\def\phi{\varphi}

\def\MB{\mbox}
\def\LD{\ldots}
\def\OV{\overline}
\def\SP#1{\langle #1 \rangle}

\begin{document}

\maketitle

\begin{abstract}
We present a simple method for the construction of polynomials with cyclic Galois groups, hoping to encourage a reader
with some background in algebra to make computations of his/her own.
\end{abstract}

\noindent

\section{Introduction}
\label{s1}

Let $K$ be a subfield of the field $\C$ of complex numbers. A randomly chosen polynomial $f\in K[X]$ of degree 5 is likely to be irreducible with the symmetric
group $S_5$ as Galois group.
So, if $x_1,\LD,x_5$ are the zeros of $f$ in $\C$, an arbitrary permutation of the zeros like
\BD
 \left(\begin{array}{ccccc}
       x_1&x_2&x_3&x_4&x_5 \\
       x_4&x_5&x_1&x_3&x_2
     \end{array}\right)
\ED
defines an automorphism of the splitting field $K[x_1,\LD,x_5]$ of $f$ over $K$. This means that a relation
\BD
   P(x_1,\LD,x_5)=0,
\ED
where $P$ is a polynomial in five variables with coefficients in $K$, remains valid if the above permutation is applied to the zeros $x_1,\LD,x_5$.

At first glance this property looks nicely. But a somewhat deeper knowledge of
algebraic equations shows that
\begin{itemize}
\item the equation $f=0$ is not solvable by radicals;
\item the splitting field $K[x_1,\LD,x_5]$ of $f$ is a vector space over $K$ of dimension $5\,!=120$;
\item the zero $x_5$, say, cannot be expressed in terms of $x_1$ alone --- which would mean that $x_5\in K[x_1]$; not even in terms of $x_1,x_2, x_3$ alone but only
in terms of $x_1$, $x_2$, $x_3$ and $x_4$.
\end{itemize}
For example, if $K=\Q$, $f = X^5+2X^4-X^3+X^2-X+1\in \Q[X]$ is such a polynomial, as the software package MAPLE easily shows.

At the other end of the scale are polynomials of degree 5 with cyclic Galois group, such as $f = X^5-5X^4-X^3+12X^2+5X-1$.
Indeed, the splitting field of $f$ over $\Q$ equals $\Q[x_1]$, which means that it is a vector space of dimension 5 over $\Q$. In particular, all of the other zeros of $f$ can be
expressed as polynomials in $x_1$ of degree (at most) 4, these polynomials being
\BD
  \begin{array}{l}
  P_2=(-18X^4+98X^3-23X^2-216X+6)/23,\\P_3=(5X^4-17X^3-46X^2+60X+98)/23,\\
  P_4=(X^4-8X^3+58X+38)/23,\\ P_5=(12X^4-73X^3+69X^2+75X-27)/23
  \end{array}
\ED
in our case.
For the time being, the reader need not know how to find these polynomials. However, one may check that they have this property, on inserting $P_2$, for instance,
in $f$,
\BD
   f(P_2)=P_2^5-5P_2^4-P_2^3+12P_2^2+5P_2-1,
\ED
and computing the residue of $f(P_2)$ modulo $f$. This residue should be zero.

As to the Galois group of $f$, i. e., the group of automorphisms of the field $\Q[x_1]$,
one can show that it equals $\SP{(x_1x_2x_3x_4x_5)}$, the group generated by the cycle $(x_1x_2x_3x_4x_5)\in S_5$. This cycle maps $x_1$ to $x_2$, $x_2$ to $x_3$,
\LD, $x_5$ to $x_1$. Here $x_j=P_j(x_1)$, $j=2,\LD,5$. Note, however, that this only works since we have numbered the polynomials $P_j$ appropriately.
Changing this numbering would mean that we had to choose a different cycle as the generator of the Galois group.

Methods for the construction of polynomials $f$ with cyclic Galois group over $K$ are known. A fundamental paper with explicit examples depending on one
parameter is \cite{De}. The method we present here has several advantages for a nonprofessional, namely,
\begin{itemize}
\item it is easy to understand and
      simple to handle if one has a software package like MAPLE at hand;
\item the interrelation of the zeros of $f$ is a priori clear; in particular, analogues of the above polynomials $P_2,\LD, P_5$ are easily found;
\item it is also easy to number the zeros in such a way that a certain cycle (like the above one) generates the Galois group.
\end{itemize}
A disadvantage of this method consists in the fact that one has to work with an appropriate extension field $K$ of $\Q$ in certain cases (instead of $\Q$).
 However, this disadvantage is compensated by the fact that the method yields polynomials over $\Q$ with other interesting Galois groups (like dihedral groups, see Section \ref{s4}).

Possibly this method is not new, but we do not know where it could be found in the literature.

What kind of knowledge does this article presume? We hope that the basic concepts of Galois theory suffice for the understanding of the method (Section \ref{s2}).
Sections \ref{s3} and \ref{s4} require a bit more. But numerous examples in the text should be helpful for feeling one's way. Of the many books on Galois theory
we list only \cite{Co, Lo, We}, each of which has its own merits.

\section{The method}
\label{s2}

As above, let $K$ be a subfield of $\C$ and $A\in GL(2,K)$, the group of invertible $2\times 2$-matrices with entries in $K$. The matrix $A$ defines
the M\"obius transform
\BD
   \C\smallsetminus K\to \C\smallsetminus K: z\mapsto A\circ z;
\ED
indeed, if
\BD
A=\left(
                \begin{array}{cc}
                  a & b \\
                  c & d \\
                \end{array}
              \right),
\MB{ then } A\circ z=\frac{az+b}{cz+d}.
\ED
Note the associativity $A\circ (B\circ z)=(AB)\circ z$, which means that we may write $A\circ B\circ z$ for this item. In particular, $A\circ A\circ \LD \circ A\circ z=A^n\circ z$,
if the matrix $A$ occurs $n$ times on the left hand side.

The main idea of the method is as follows.
Suppose we are given a zero $x_1$ of an irreducible polynomial $f\in K[X]$ of degree $n\ge 3$. Then one may hope that the other zeros are
\BD
  x_2=A\circ x_1, x_3=A^2\circ x_1, \LD, x_n=A^{n-1}\circ x_1.
\ED
In order that the Galois group of $f$ equals $\SP{(x_1,\LD,x_n)}$, one must have $A^n\circ x_1=x_1$.
Let
\BE
\label{2.5}
A^n=\left(
           \begin{array}{cc}
             a' & b' \\
             c' & d' \\
           \end{array}
         \right).
\EE
Then $c'x_1^2+(d'-a')x_1-b'=0$ holds. In the case $c'\ne 0$, $x_1$ is a zero of a quadratic polynomial in $K[X]$ and
of the irreducible polynomial $f$ of degree $n\ge 3$. This is impossible. Hence $c'=0$, but then $d'-a'=0$ and $b'=0$.
Accordingly, $A^n$ is the diagonal matrix $a'\cdot I_2$,
where $I_2$ is the $2\times 2$-unit matrix.

How to find such a matrix $A$? For this purpose we put
\BD
 N=\left\{
     \begin{array}{ll}
      n , & \MB{ if }n \MB{ is odd;} \\
      2n , & \hbox{ if }n \MB{ is even.}
     \end{array}
   \right.
\ED
Let $\zeta\in \C$ be a primitive $N$th root of unity. In most cases we choose
$\zeta=\zeta_N=e^{2\pi i/N}$. We need a $2\times 2$-matrix $A=\left(
                                                             \begin{array}{cc}
                                                               a & b \\
                                                               c & d \\
                                                             \end{array}
                                                           \right)$
with entries in $K$ that is similar to
\BE
\label{2.2}
   \left(
     \begin{array}{cc}
       r\zeta & 0 \\
       0 & r\zeta^{-1}\\
     \end{array}
   \right) \MB{ for some } r\in K.
\EE
Similarity here means similarity over $\C$.
Since the matrix in (\ref{2.2}) has two different eigenvalues, it suffices that $A$ has the same characteristic polynomial, i. e.,
the same trace and the same determinant. So we have the conditions
\BE
\label{2.3}
a+d=r(\zeta+\zeta^{-1})\MB{ and } A= ad-bc= (r\zeta)(r\zeta^{-1})=r^2
\EE
for the entries $a$, $b$, $c$, $d$.
Because the matrix $A^n$ is similar to $\pm r^n\cdot I_2$, it must coincide with this multiple of the unit matrix. Hence it has the desired property.

\medskip
\noindent
{\em Example.}
Let $n=6$, so $N=12$. We obtain, for $\zeta=\zeta_{12}$, $\zeta+\zeta^{-1}=2\cos(2\pi/12)=\sqrt 3$, so one possibly thinks that one has to choose $K=\Q[\sqrt 3]$. Here, however,
$r$ plays a fruitful role, since one can choose $r=\sqrt 3$, so $A$ may be a rational $2\times 2$-matrix with $a+d=\sqrt 3\cdot\sqrt 3=3$ and $ad-bc=\sqrt 3^2=3$.
Accordingly, we may choose $A=\left(
                                \begin{array}{cc}
                                  1 & 1 \\
                                  -1 & 2 \\
                                \end{array}
                              \right)$, say.

\medskip
How to find the polynomial $f$ whose roots are $x_1,\LD,x_n$? For this purpose we consider the rational function field
\BD
  K(X)=\left\{\frac{f_1}{f_2}: f_1,f_2\in K[X], f_2\ne 0\right\}.
\ED
For $A$ as above and $g\in K(X)\smallsetminus K$ we
define
\BD
  A\circ g=\frac{ag+b}{cg+d},
\ED
which lies in $K(X)\smallsetminus K$. As in the case of $\C\smallsetminus K$ we have the relation
\BD
  A\circ A\circ \LD\circ A\circ g=A^n\circ g,
\ED
where $A$ occurs $n$ times on the left hand side. Our candidate for $f$ is a numerator polynomial $f_1$ of the rational function
\BE
\label{2.4}
 F=X+A\circ X+A^2\circ X+\LD+A^{n-1}\circ X+ C=\frac{f_1}{f_2},
\EE
where $f_1,f_2\in K[X]$ are co-prime polynomials. The constant $C\in K$ allows a variation of $F$. Such a variation can be necessary for making $f_1$ irreducible.
The polynomial $f_1$ is uniquely determined only up to a factor in $K\smallsetminus\{0\}$.

\medskip
\noindent
{\em Example.} Let $A$ be as in the above example. We obtain, with $C=1$,
\BD
  F= X+\frac{X+1}{-X+2}+\frac 3{-3X+3}+\frac{-3X+6}{-6X+3}-\frac{-9X+9}{9X}+\frac{-18X+9}{-9X-9}+1,
\ED
whose numerator polynomial
\BD
   f_1 = 2X^6+2X^5-35X^4+40X^3+5X^2-14X+2
\ED
is irreducible over $\Q$. Indeed, it is a polynomial with Galois group $\SP{(x_1x_2x_3x_4x_5x_6)}$,
where $x_1$ is an arbitrary zero of $f_1$ and $x_j=A^{j-1}\circ x_1$ for $j=2,\LD,6$.
This follows from the following more general result.

\begin{theorem} % Theorem 1 %%%%%%%%%%%%%%%%%%%%%%%%%%%%%%%%%%%%%%%%%%%%%%%%%%%%%%%%%%%%%%%%%%%%%%%%%%%
\label{t1}

Let $A\in GL(2,K)$ have the above properties, in particular, $A$ satisfies {\rm (\ref{2.3})}. Let $n\ge 3$, $F$ be defined as in {\rm (\ref{2.4})} and $f_1$ a numerator
polynomial of $F$. Suppose that $f_1$ is irreducible and has degree $n$. For an arbitrary zero $x_1$ of $f_1$ in $\C$, put $x_j=A^{j-1}\circ x_1$,
$j=2,\LD,n$. Then the numbers $x_1,\LD,x_n$ are the complex zeros
of $f_1$. Further, $K[x_1]$ is a Galois extension of $K$ with Galois group $\SP{(x_1x_2\LD x_n)}$.

\end{theorem} %%%%%%%%%%%%%%%%%%%%%%%%%%%%%%%%%%%%%%%%%%%%%%%%%%%%%%%%%%%%%%%%%%%%%%%%%%%%%%%%%%%%%%%%%

\begin{proof}
The numbers $x_2,\LD,x_n$ are zeros of $f_1$. Indeed, for $x_2=A\circ x_1$ we have
\BD
  F(x_2)= A\circ x_1+A^2\circ x_1+\LD+A^n\circ x_1+C.
\ED
However, since $A^n\circ x_1=x_1$, this is the same as $F(x_1)$, which is zero (observe that $f_2(x_1)\ne 0$ for the denominator polynomial $f_2$ of $F$ since
$f_1$ and $f_2$ are co-prime). This argument also works
for $x_3,\LD,x_n$.

Suppose that the numbers $x_1,\LD,x_n$ are not pairwise different, so
$x_j=x_k$ for $1\le j<k\le n$. Put $l=k-j\in\{1,\LD,n-1\}$.
Then $A^l\circ x_1=x_1$. As in the context of (\ref{2.5}) we conclude that
$A^l$ is a diagonal matrix $t\cdot I_2$, $t\in K$. Since $A^l$ is similar to
\BD
   \left(
     \begin{array}{cc}
       r^l\zeta^l & 0 \\
       0 & r^l\zeta^{-l} \\
     \end{array}
   \right),
\ED
we obtain $r^l\zeta^l=t=r^l\zeta^{-l}$. Therefore, $\zeta^{2l}=1$, which is only possible if $N$ divides $2l$. This contradicts $l<n$
(if $n=N$ is odd, $2l$ is different from $N$).

Now we know that $\{x_1,\LD,x_n\}$ is the set of all complex zeros of $f_1$.
Obviously, $K[x_1]$ is the splitting field of
$f_1$, since it contains $x_1,\LD,x_n$. In particular, it is a Galois extension of $K$. Let $k\in\{1,\LD,n\}$. Since $f$ is irreducible,
an automorphism of the field $K[x_1]$ over $K$ can be defined by
\BD
 K[x_1]\to K[x_1]: x_1\mapsto x_k=A^{k-1}\circ x_1
\ED
(see \cite[Cor. 2.6.2]{We}).
Hence we are given $n$ distinct automorphisms of this kind, which form the Galois group of $K[x_1]$ over $K$. Considered as a permutation of $x_1,\LD,x_n$,
the automorphism $x_1\mapsto x_2$ equals $(x_1x_2\LD x_n)$, since it maps $x_2=A\circ x_1$ to $A\circ x_2=x_3$, and so on. It is easy to see that $x_1\mapsto x_3$
is just the permutation $(x_1\LD x_n)^2$. We obtain the whole Galois group in this way.
\end{proof}

\begin{remark}
It is not difficult to see that a numerator polynomial $f_1$ of $F$ has always a degree $\le n$. The degree $n$ is just what one expects if $A$ is
not of a very special form.

\end{remark}

\noindent
{\em Example.} Let $n=N=5$ and $\zeta=\zeta_5$. We put $r=1$. Our matrix $A=\left(
                                                                        \begin{array}{cc}
                                                                          a & b \\
                                                                          c & d \\
                                                                        \end{array}
                                                                      \right)$
has to satisfy $a+d=\zeta+\zeta^{-1}=2\cos(2\pi/5)=(-1+\sqrt 5)/2$ and $ad-bc=1$. Hence we put $K=\Q[\sqrt 5]$ and
\BD
 A=\left(
     \begin{array}{cc}
       1 & (-5+\sqrt 5)/2 \\
       1 & (-3+\sqrt 5)/2 \\
     \end{array}
   \right).
\ED
If we choose $C=0$, a numerator polynomial of $F$ is
\begin{eqnarray}
\label{2.6}
  f_1=(3-\sqrt 5)X^5+(-3+\sqrt 5)X^4+(-80+32\sqrt 5)X^3\nonumber \\
   +(300-128\sqrt 5)X^2+(-395+175\sqrt 5)X+178-80\sqrt 5.
\end{eqnarray}
MAPLE is able to factorize $f_1$ over the field $K$. It says that $f_1$ is irreducible (the proof of Theorem \ref{t2} yields a different method to show this).
Starting with the zero $x_1\approx -3.585182$ of $f_1$, we obtain $x_2=A\circ x_1\approx 1.252070$, $x_3=A\circ x_2\approx -0.149288$, $x_4=A\circ x_3\approx 2.882340$, and
$x_5=A\circ x_4\approx 0.600060$ as the other zeros of $f_1$.

In Section 4 we will see that $f_1$ gives rise to an irreducible polynomial in $\Q[X]$ of degree 5 whose splitting field has a Galois
group isomorphic to the dihedral group $D_5$ of order $10$.

\begin{remark}
In the setting of Theorem \ref{t1} it is easy to find a polynomial $P\in K[X]$ such that $x_2=P(x_1)$. We may assume that $c\ne 0$,  since otherwise $F$ (and $f_1$) is a polynomial of degree $\le 1$.
As $f_1$ and $cX+d$ are co-prime, the
extended Euclidean algorithm yields polynomials $g$ and $h$ in $K[X]$
such that
\BD
   f_1g+(cX+d)h=1.
\ED
Because $f_1(x_1)=0$, we have $(cx_1+d)h(x_1)=1$ and $A\circ x_1=(ax_1+b)h(x_1)$. The remainder of $(aX+b)h$ modulo $f_1$ is a polynomial $P$ of degree $\le n-1$ of the desired kind. In the case of the above example, MAPLE yields
\BEA
 P= -X^4+(-1+\sqrt 5) X^3/2+(18-3\sqrt 5) X^2+\\
  (-61+15\sqrt 5)X/2+14-6\sqrt 5.
\EEA
\end{remark}

Next we show how Theorem \ref{t1} works for $n=8$, $10$ and $12$ provided that $K$ is a suitable quadratic extension of $\Q$.
In all of these cases $N=2n$. Here $\zeta+\zeta^{-1}$ does not lie in the quadratic extension $K$ of $\Q$, but $(\zeta+\zeta^{-1})^2$ does.
Accordingly, we put
\BD
   r=\frac{s}{\zeta+\zeta^{-1}}
\ED
for a number $s\in K$. Then condition (\ref{2.3}) reads
\BE
\label{2.8}
   a+d=s,\enspace ad-bc=\frac{s^2}{(\zeta+\zeta^{-1})^2}.
\EE
For $n=8$ we have, with $\zeta=\zeta_{16}$, $(\zeta+\zeta^{-1})^2=\zeta_8+2+\zeta_8^{-1}=2+2\cos(\pi/4)=2+\sqrt 2$.
Then (\ref{2.8}) says
\BD
   a+d=s,\enspace ad-bc=\frac{s^2}{2+\sqrt 2}=\frac{s^2(2-\sqrt 2)}2.
\ED
Hence we put $K=\Q[\sqrt 2]$. One may choose $s=2$, say, but $s=2+\sqrt 2$ leads to a slightly nicer result.
So we have $s^2/(2+\sqrt 2)=2+\sqrt 2$. A matrix $A$ with this property is
\BD
  A=\left(
      \begin{array}{cc}
        2 & -2+\sqrt 2 \\
        1 & \sqrt 2 \\
      \end{array}
    \right).
\ED
If we choose $C=1$, we obtain a numerator polynomial $f_1\in K[X]$ with the properties of Theorem \ref{t1}.

In the case $n=10$ we choose $\zeta=\zeta_{20}^3$, which gives a nicer result than $\zeta_{20}$ itself. We have $(\zeta+\zeta^{-1})^2=2+2\cos(3\pi/5)=(5-\sqrt 5)/2$.
Then (\ref{2.8}) reads
\BD
  a+d=s,\enspace ad-bc=s^2(5+\sqrt 5)/10.
\ED
If we choose $s=10$, a suitable matrix $A$ is
\BD
  A=\left(
      \begin{array}{cc}
        5 & 5 \\
        -5-2\sqrt 5 & 5 \\
      \end{array}
    \right),
\ED
and on putting $C=1$ we obtain a polynomial $f_1$ of the desired kind.

In the case $n=12$ we choose $\zeta=\zeta_{24}$ and have $(\zeta+\zeta^{-1})^2=2+2\cos(\pi/6)=2+\sqrt 3$.
By (\ref{2.8}),
\BD
a+d=s,\enspace ad-bc=s^2(2-\sqrt 3).
\ED
If we choose $s=1$, a suitable matrix is
\BD
 A=\left(
     \begin{array}{cc}
       2 & -4+\sqrt 3 \\
       1 & -1 \\
     \end{array}
   \right)
\ED
With $C=1$ this gives a polynomial $f_1$ of the desired kind.

\begin{remark}
In the setting of (\ref{2.3}) suppose that $K$ is a real field. Hence the determinant $r^2$ of $A$ is positive.
This, however, means that the M\"obius transform $z\mapsto A\circ z$ maps the upper half-plane $\{u+iv: u,v\in\R, v>0\}$
into itself; and the same holds for the lower half-plane. Therefore, all complex zeros of $f_1$ must be real. Indeed, suppose
that $x_1\not\in\R$. Since the complex-conjugate $\OV{x_1}$ of $x_1$ is a zero of $f_1$, we have $\OV{x_1}=x_k$ for some $k$. But $x_k=A^{k-1}\circ x_1$ lies in the same half-plane
as $x_1$, which is impossible for $\OV{x_1}$. In particular, we have only real zeros in the above cases $n=8,
10, 12$.
\end{remark}

We conclude this section with a look at the parameters which we can choose if we construct the polynomial $f_1$ in the above way. The parameter $r$ has only the purpose
of diminishing the degree of $K$ over $\Q$. If we multiply this parameter by a non-vanishing element of $K$, we obtain the same set of polynomials $F$.
Hence $r$ cannot be considered as a free parameter. By (\ref{2.3}), however,
two of the four parameters $a,b,c,d$ are free, and the constant $C$ is also free. Therefore, we can vary three quantities in $K$ in order to obtain an irreducible polynomial $f_1$
of degree $n$.

\section{A wreath product}
\label{s3}
In the above cases $n=5,8,10, 12$ we do not find polynomials $f_1$ with rational coefficients but with coefficients in a quadratic extension $K$ of $\Q$. How can we
obtain a polynomial in $\Q[X]$ from such an $f_1$? Here the multiplication with the conjugate polynomial $f_1'$ is an obvious way.

In order to make this idea precise, we suppose that $K$ equals $\Q[\sqrt D]$ for $D\in\{2,3,5\}$. Let $f_1\in K[X]$ be given and
\BD
   (\enspace)':K\to K:u+v\sqrt D\mapsto u-v\sqrt D
\ED
be the non-trivial field automorphism of $K$. The map $(\enspace)'$ has a natural extension to the polynomial ring $K[X]$. It maps each coefficient $\alpha$ of
a polynomial to the respective coefficient $\alpha'=(\enspace)'(\alpha)$. We denote this extension also by $(\enspace)'$. For example, the polynomial $f_1$ of (\ref{2.6})
is mapped to the conjugate polynomial
\BEA
  f_1'=(3+\sqrt 5)X^5+(-3-\sqrt 5)X^4+(-80-32\sqrt 5)X^3\\
   +(300+128\sqrt 5)X^2+(-395-175\sqrt 5)X+178+80\sqrt 5.
\EEA
Put $f=f_1\cdot f_1'$. Then this polynomial lies in $\Q[X]$. In our case, we obtain
\BEA
  f=4(X^{10}-2X^9-39X^8+170X^7+35X^6-1538X^5\\
  +3753X^4-3970X^3+1825X^2-155X-79).
\EEA
The polynomial $f$ is irreducible.
What is the Galois group of the splitting field of $f$ over $\Q$? It turns out that it is isomorphic to the wreath product $C_5\wr C_2$. Roughly speaking, this means the following:
Let $C_5$ denote a cyclic group of order $5$. Then $C_5\wr C_2$ contains the cartesian product $C_5\times C_5$ as a normal subgroup of index 2. In particular,
$|C_5\wr C_2|=5\cdot 5\cdot 2=50$ and $C_5\wr C_2$ contains a cyclic group $C_2$ of order $2$ that acts on this normal subgroup in a certain way. We will explain this action
in greater detail in Theorem \ref{t3}.

The proof that the Galois group has this structure is fairly easy: One factorizes $f$ modulo a number of primes until the factorization type $(5,1,1,1,1,1)$ occurs, i. e.,
there is an irreducible factor of degree 5 and five factors of degree 1, each of which occurs with multiplicity 1. In our case $29$ is such a prime.

But all these things will be explained in the next two theorems. In these theorems we need not assume that $f_1$ or $f_1'$ is irreducible in $K[X]$.
It suffices that $f_1f_1'$ is irreducible in $\Q[X]$.

\begin{theorem} %Theorem 2%%%%%%%%%%%%%%%%%%%%%%%%%%%%%%%%%%%%%%%%%%%%%%%%%%%%%%%%%%%%
\label{t2}
Let $n\in\{ 5, 8, 10, 12\}$ and $K$ the corresponding quadratic extension of $\Q$. Suppose that the polynomial $f_1\in K[X]$ of Theorem \ref{t1} has degree $n$.
Let $f_1'\in K[X]$ be the conjugate polynomial of $f_1$ and $f=f_1f_1'$ be irreducible in $\Q[X]$.
Let $m$ be a natural number such that $mf$ lies in $\Z[X]$. Suppose that, for some prime number $p$,
$mf$ has a factorization of type $(n, 1,\LD,1)$ modulo $p$. Then $L=K[x_1,y_1]$ is the splitting field of $f$ over $\Q$,
where $f_1(x_1)=0=f_1'(y_1)$. The degree of $L$ over $\Q$ is
\BD
   [L:\Q]=2n^2.
\ED
\end{theorem} %%%%%%%%%%%%%%%%%%%%%%%%%%%%%%%%%%%%%%%%%%%%%%%%%%%%%%%%%%%%%%%%%%%%%%%%

\begin{proof}
Since $f$ is irreducible over $\Q$, $f_1$ must be irreducible in $K[X]$. Otherwise, we have a non-trivial factorization $f_1=g\cdot h$ in $K[X]$, which produces the corresponding
factorization $f_1'=g'h'$ of the conjugate polynomial. But then $gg'$ and $hh'$ lie in $\Q[X]$ and $f$ has a non-trivial factorization over $\Q$, which we have excluded.
In the same way $f_1'$ is irreducible in $K[X]$.

As we assume, $x_1\in\C$ is a zero of $f_1$, so $x_2=A\circ x_1, \LD ,x_n=A^{n-1}\circ x_1$ are the remaining zeros of $f_1$, by Theorem \ref{t1}. By assumption, $y_1\in\C$ is a zero of $f_1'$. Now $f_1'$ is a numerator polynomial of
\BD
 F'=X+A'\circ X+A'^2\circ X+\LD+A'^{n-1}\circ X+C',
\ED
where $A'$ arises from $A$ if we apply the map $(\enspace)'$ to the entries of $A$.
As in the case of $f_1$ and $x_1$,  we see that $y_2=A'\circ y_1, \LD, y_n=A'^{n-1}\circ y_1$ are the zeros of $f_1'$.
Hence $L=K[x_1,y_1]$ is the splitting field of $f$ over $K$.

We show that $L$ equals $\Q[x_1,\LD,x_n,y_1,\LD,y_n]$, the splitting field of
$f$ over $\Q$. For this purpose we assume that $f_1$ is monic (recall that it is unique only up to a factor in $K\smallsetminus\{0\}$).
Then $f_1$ does not lie in $\Q[X]$. Indeed, if $f_1$ were in $\Q[X]$, then $f_1'$ would be equal to $f_1$, and
$f=f_1^2$ would be reducible.
So $f_1$ must have a coefficient in $K\smallsetminus \Q$. This coefficient generates the quadratic field $K$ over $\Q$, and it has the form
$P(x_1,\LD,x_n)$ for a (symmetric) polynomial $P$ in $n$ variables with coefficients in $\Q$.
Accordingly, $K[x_1]=\Q[x_1,\LD,x_n]$ and, in the same way, $K[y_1]=\Q[y_1,\LD,y_n]$. Altogether, $L=K[x_1,y_1]=\Q[x_1,\LD,x_n,y_1,\LD,y_n]$.

Next we consider the tower of fields
\BD
  \Q\subseteq K\subseteq K[x_1]\subseteq K[x_1, y_1]=L
\ED
and obtain the degrees $[K:\Q]=2$, $[K[x_1]:K]=n$ and $[K[x_1,y_1]:K[x_1]]\le [K[y_1]:K]=n$. Therefore,
$[L:K]\le 2\cdot n\cdot n=2n^2$.

Now $mf\in \Z[X]$ has a factorization of type $(n, 1,\LD,1)$ mod $p$ for some prime $p$. This means that
there exists an automorphism $\sigma$ of order $n$ in the Galois group of $L$ over $\Q$ that fixes $n$ zeros of $f$
(the existence of $\sigma$ is a consequence of Tchebotarev's density theorem, see \cite{StLe} or \cite[p. 402]{Co}).
Let $z$ be such a zero. Then $\sigma$ also fixes $\Q[z]$, and $[L:\Q[z]]$ must be $\ge n$, the order of $\sigma$.
On the other hand,
$[\Q[z]:\Q]=2n$, since $f$ is irreducible.
Altogether, we obtain $[L:\Q]\ge n\cdot 2n=2n^2$. Both estimates together show $[L:\Q]=2n^2$.
\end{proof}

 The field $L$ is the composite of the Galois extensions $K_1=K[x_1]$ and $K_2=K[y_1]$ of $K$, whose Galois groups
are $\SP{(x_1\LD x_n)}$ and $\SP{(y_1\LD y_n)}$ (see \cite[Sect. 2, F11]{Lo}. It is well-known that the map
\BD
  \sigma\mapsto (\sigma|_{K_1},\sigma|_{K_2})
\ED
embeds the Galois group of $L$ over $K$ in the cartesian product $\SP{(x_1\LD x_n)}\times \SP{(y_1\LD y_n)}$ (see \cite[Sect. 12, F2]{Lo}).
Since both groups have the same number of elements
(which equals $n^2=[L:K]$), this embedding must be an isomorphism. Hence the Galois group of $L$ over $K$ contains an element $\sigma$ such that
$\sigma|_{K_1}=(x_1\LD x_n)$, $\sigma|_{K_2}=id$, and an element $\tau$ such that $\tau|_{K_1}=id$ and $\tau|_{K_2}=(y_1\LD y_n)$.

Let $G$ be the Galois group of $L$ over $\Q$. For the sake of simplicity we denote $\sigma\in G$ simply by $(x_1\LD x_n)$  and $\tau\in G$ by $(y_1\LD y_n)$.
The group $\SP{\sigma, \tau}$ of order $n^2$ is just the Galois group of $L$ over $K$. Let $\rho\in G$ be such that $\rho|_K=(\enspace)'$. Such a $\rho$ exists,
since $K$ is a Galois extension of $\Q$, whose Galois group arises from $G$ by restriction: If we consider $\mu|_K$ for all $\mu\in G$, we must obtain all
elements of the Galois group of $K$ over $\Q$ (see \cite[Sect. 8, F4]{Lo}).

Note that $\rho(x_1)$ must be one of $y_1,\LD, y_n$, since
\BD
  0=f_1(x_1)=\rho(f_1)(\rho(x_1))=f_1'(\rho(x_1)).
\ED
As above, $\rho(f_1)$ means that $\rho$ is applied to all coefficients of $f_1$. The same argument shows that $\rho(y_1)$ is one of $x_1,\LD,x_n$. This implies
that, for suitably chosen exponents $j$, $k$, the automorphism $\sigma^j\tau^k\rho$ maps $x_1$ to $y_1$ and $y_1$ to $x_1$. For reasons of simplicity
we write $\rho$ instead of $\sigma^j\tau^k\rho$. So $\rho$ exchanges $x_1$ and $y_1$. But then it also exchanges $x_j$ and $y_j$ for $j=2,\LD,n$,
since we have
\BD
\rho(x_j)=\rho(A^{j-1}\circ x_1)=\rho(A^{j-1})\circ\rho(x_1)=A'^{j-1}\circ y_1=y_j.
\ED
Altogether, we may write $\rho=(x_1y_1)(x_2y_2)\LD(x_ny_n)$. Hence we have shown the following theorem.

\begin{theorem}% Theorem 3 %%%%%%%%%%%%%%%%%%%%%%%%%%%%%%%%%%%%%%%%%%%%%%%%%%%%%%%%%%%%%%%%%%%%%%
\label{t3}
Under the assumptions of Theorem {\rm\ref{t2}}, the Galois group $G$ of $L$ over $K$ is generated by
$\sigma=(x_1\LD x_n)$, $\tau=(y_1\LD y_n)$, and $\rho=(x_1y_1)\LD(x_ny_n)$. It contains the normal subgroup
$\SP{\sigma,\tau}$ of order $n^2$ and index 2.
Further $G=\{\sigma^j\tau^k\rho^l: j,k=0,\LD,n-1, l=0,1\}$.

\end{theorem} %%%%%%%%%%%%%%%%%%%%%%%%%%%%%%%%%%%%%%%%%%%%%%%%%%%%%%%%%%%%%%%%%%%%%%%%%%%%%%%%%%%%%%%

\begin{remark}
In the context of the theorem, the groups $\SP{\sigma}$ and $\SP{\tau}$ are cyclic subgroups of $G$, isomorphic to the abstract cyclic group $C_n$
of order $n$. Then $\SP{\sigma,\tau}$ is a group isomorphic to the cartesian product $C_n\times C_n$. Moreover, it is a normal subgroup of index $2$
in $G$, and $\SP{\rho}$ is isomorphic to $C_2$.
In addition, we have the composition rules
\BE
\label{3.4}
 \rho\sigma= \tau\rho,\rho\tau=\sigma\rho.
\EE
This is what the notation $G\cong C_n\wr C_2$ (wreath product of $C_n$ by $C_2$) says.
\end{remark}

The examples of the foregoing section for $n=5, 8, 10, 12$ all lead to polynomials over $\Q$ with Galois group $C_n\wr C_2$.
For instance, if $n=10$, we obtain
\BEA
f=X^{20}+2X^{19}-449X^{18}-570X^{17}+28905X^{16}+26568X^{15}\\-648012X^{14}-371400X^{13}+5691930X^{12}
+2011900X^{11}\\ -19618774X^{10}-4515500X^9+28459650X^8+4137000X^7-16200300X^6\\
-1473000X^5+3613125X^4+191250X^3-280625X^2-6250X+3125,
\EEA
where we have divided our original $f$ by $-2\cdot 10^9$.

\section{Dihedral groups}
\label{s4}

Let $C_n=\SP{\sigma}$ be a cyclic group of order $n$. The dihedral group $D_n$ of order $2n$ can be defined as the group generated by $\sigma$ and an element $\rho$ of
order $2$ such that
\BE
\label{4.2}
\sigma\rho=\rho \sigma^{-1}.
\EE
Following \cite{MaSc}, we construct a polynomial $g\in\Q[X]$ of degree $n$ whose Galois group is $D_n$ for $n=5,8,10$ and $12$.

Let $f=f_1f_1'$ be as in the foregoing section. In particular, $x_1,\LD,x_n$ are the zeros of
$f_1$, $y_1,\LD,y_n$ the zeros of $f_1'$. Moreover, $L=\Q[x_1,\LD,x_n,y_1,\LD, y_n]$ is the splitting field of $f$ over $\Q$, whose Galois group $G$ is generated by $\sigma=(x_1\LD x_n)$, $\tau=(y_1\LD y_n)$ and $\rho=(x_1y_1)\LD (x_ny_n)$. Let $H=\SP{\sigma\tau}$, a cyclic subgroup of $G$ of order $n$. By (\ref{3.4}), $\rho\sigma\rho=\tau$ and $\rho\tau\rho=\sigma$,
so we have
\BE
\label{4.4}
\rho\sigma\tau\rho=\sigma\tau.
\EE
Further, $\sigma$ and $\tau$ commute with $\sigma\tau$. These facts show that $H$ is a normal subgroup of $G$,
the factor group $G/H$ being of order $2n$.

Observe that $\sigma(\sigma\tau)^{-1}=\tau^{-1}$, so we obtain
$\OV{\sigma}=\OV{\tau}^{-1}$
for the corresponding residue classes in $G/H$. From (\ref{3.4}) we see that $\sigma\rho=\rho\tau$, hence
\BD
\OV{\sigma}\,\OV{\rho}=\OV{\rho}\,\OV{\sigma}^{-1}.
\ED
In other words, $\OV{\sigma}$ and $\OV{\rho}$ satisfy the relation (\ref{4.2}). Accordingly, $G/H$ is (isomorphic to) the group $D_n$. Since $H$ is normal in $G$,
the fixed field $M$ of $H$ is a Galois extension of $\Q$ with Galois group isomorphic to $G/H$.
In particular, $[M:\Q]=2n$.

For our purpose, however, we need a subfield $M'$ of $M$ with $[M':\Q]=n$. Therefore, we look at the
relation (\ref{4.4}), which shows that the group $H'=\SP{H,\rho}=\SP{\sigma\tau,\rho}$ is abelian of order $2n$. However, it is not a normal subgroup of $G$,
since $\sigma\rho\sigma^{-1}=\rho\tau\sigma^{-1}=\rho\sigma^{-1}\tau\not\in H'$. Hence the fixed field of $H'$ is a subfield $M'$ of $M$ that is not Galois over $\Q$.
Since $[M:M']=[H':H]=2$, the Galois closure of $M'$ must be $M$.

In order to generate the field $M'$, we look for an element $z\in L$ that remains fixed under $H'$. Such an element is, for instance,
\BE
\label{4.6}
 z=x_1y_1+x_2y_2+\LD+x_ny_n.
\EE
But this element generates $M'$ only if $H'$ is the set of {\em all} elements of the Galois group that fix $z$. In other words, for an element $\mu\in G\smallsetminus H$
we should have $\mu(z)\ne z$. This assertion can be phrased in a slightly different way: If $R$ is a system of representatives of $G/H'$, the numbers $\mu(z)$
should be distinct for $\mu\in R$. In our case the group $\SP{\sigma}$ itself is a suitable set $R$.
So if the numbers $z$, $\sigma(z)$, \LD, $\sigma^{n-1}(z)$ are pairwise different, we have $M'=\Q[z]$. Moreover,
\BD
   g=\prod_{j=0}^{n-1}(X-\sigma^j(z))
\ED
is the minimal polynomial of $z$ over $\Q$, a polynomial whose splitting field has the Galois group $D_n$.

Possibly the simplest way to check whether $z$ is a suitable choice consists in approximating the elements $z,\sigma(z),\sigma^{n-1}(z)$ numerically.
Provided that the monic polynomial that belongs to $f$ has integer coefficients, the numbers $z$, $\sigma(z), \LD, \sigma^{n-1}(z)$ are algebraic integers.
This implies that $g$ is a polynomial with integer coefficients. Hence a sufficiently precise numerical approximation of these coefficients exhibits $g$.

\medskip
\noindent
{\em Examples.} Let $n=5$ and $f$ be as in the example preceding Theorem \ref{t2}. Obviously, the corresponding monic polynomial has integer coefficients.
With $z$ as in (\ref{4.6}), we obtain $g=X^5-X^4-840X^3-6135X^2+24775X-19900$. Here we have used 20 digits of precision for the zeros of $f$. The numerical result for $g$ is nearly as precise. To be sure, we determine the Galois group of the splitting field of $g$ with the help of MAPLE. It turns out that it is $D_5$.

In the case $n=10$ we look at the example at the end of the foregoing section. We find
\BEA
 g=X^{10}-X^9-22950X^8+6120X^7+166823200X^6-61635600X^5\\-383129676000X^4
+257616832000X^3+23507471680000X^2\\+13131916800000X-235594368000000
\EEA
as the minimal polynomial of $z$. MAPLE does not supply the Galois group of polynomials of a degree $\ge 10$, but factorizing $g$ modulo a number of primes produces only the types
$(10), (5,5), (2,2,2,2,2), (2,2,2,2,1,1)$ and $(1,\LD,1)$, which are
admissible for $D_{10}$ (recall the proof of Theorem \ref{t2} and the references \cite{StLe} and \cite[p. 402]{Co}).
Hence it is at least plausible that our computations are correct.

What can we do if the number $z$ of (\ref{4.6}) fails to generate $M'$ since the conjugates $\sigma^j(z)$, $j=0,\LD,n-1$ are not pairwise different?
In this case one can work with numbers of the same structure like
\BD
  x_1^2y_1^2+\LD +x_n^2y_n^2.
\ED

\begin{remark} The method presented here is not restricted to the ground fields $\Q$ or $\Q[\sqrt D]$ for $D=2,3,5$.
In the case $n=7$, for instance, a possible ground field $K$ is $\Q[v]$ for $v=\zeta_7+\zeta_7^{-1}=2\cos(2\pi/7)$, whose minimal polynomial is
$X^3+X^2-2X-1$. MAPLE allows computations in the cubic field $K$. For instance, if we take
\BD
  A=\left(
      \begin{array}{cc}
        1 & 1 \\
        v-2 & v-1 \\
      \end{array}
    \right)
\ED
and $C=1$, we arrive at an irreducible polynomial in $f_1\in K[X]$ of degree 7. By Theorem \ref{t1}, we obtain an extension $K[x_1]$ of $K$ with cyclic Galois group
of order 7.
\end{remark}

\centerline{\bf Acknowledgment.}
The author wishes to thank Gerhard Kirchner for a helpful discussion.

\vfill\eject

\end{document}